\documentclass[11pt]{article}
\usepackage{graphicx}
\usepackage[a4paper,margin=1in]{geometry}
\usepackage[T1]{fontenc}
\usepackage[utf8]{inputenc}
\usepackage{lmodern}
\usepackage{microtype}
\usepackage{subcaption}
\usepackage{epigraph}
\usepackage{ragged2e}
\usepackage{amsmath,amssymb,amsthm,mathtools, centernot}
\usepackage{bbm}
\usepackage{enumitem}
\usepackage{hyperref}
\usepackage[nameinlink,capitalize]{cleveref}

\hypersetup{
    colorlinks=true,
    linkcolor=blue,
    citecolor=blue,
    urlcolor=blue
}

\theoremstyle{plain}
\newtheorem{theorem}{Theorem}[section]
\newtheorem{proposition}[theorem]{Proposition}
\newtheorem{lemma}[theorem]{Lemma}
\newtheorem{corollary}[theorem]{Corollary}

\theoremstyle{definition}

\newtheorem{assumption}[theorem]{Assumption}

\theoremstyle{remark}
\newtheorem{remark}[theorem]{Remark}

\title{On the asymptotic behaviour of $\alpha$-divergence generalized posteriors}
\author{
    Ioar Casado-Telletxea\\
    Centre for AI Fundamentals, University of Manchester\\
    \texttt{ioar.casadotelletxea@manchester.ac.uk}
    \and
    Omar Rivasplata\\
    Centre for AI Fundamentals, University of Manchester\\
    \texttt{omar.rivasplata@manchester.ac.uk}
}

\date{\today}

\usepackage{xspace}
\usepackage[colorinlistoftodos
,color=blue!30!white
]{todonotes}

\begin{document}

\maketitle

\begin{abstract}
    We study the asymptotic behaviour of the solutions to the unconstrained generalized variational inference problem regularized by an $\alpha$-divergence.
    Under mild conditions, the minimizer is explicit up to a scalar normalizing constant, and we show that it admits a rescaling under which it converges in $L^1$ to a fixed non-Gaussian law, at a concentration scale around the empirical loss minimizer that departs from the typical $n^{-1/2}$ in the classical Bernstein--von Mises theorem. For $\alpha\in(0,1)$ and $d<2/(1-\alpha)$ the limit is a Student-$t$ distribution and the posterior width shrinks faster than the sampling error of the empirical minimizer. Conversely, for $\alpha>1$ the limiting distribution is a compactly supported Pearson type-II distribution, and the posterior width shrinks more slowly than the sampling error. The KL divergence ($\alpha=1$) is the unique member of the $\alpha$-divergence family whose posterior width matches the sampling scale and whose limit is Gaussian.
    We discuss the implications of these qualitatively different regimes for the frequentist coverage of credible regions, and a toy example on intercept-only logistic regression illustrates the theoretical claims.
\end{abstract}

\vskip 1cm

\tableofcontents

\section{Introduction}

\begingroup
\setlength{\epigraphwidth}{0.80\textwidth}
\setlength{\epigraphrule}{0pt}
\setlength{\beforeepigraphskip}{1.5\baselineskip}
\setlength{\afterepigraphskip}{2\baselineskip}
\renewcommand{\epigraphsize}{\small}

\vskip -0.1cm

\epigraph{\justifying%
    \itshape
    Studying only the model theory of first-order logic would be
    analogous to the study of real analysis never knowing of any
    but the polynomial functions: core concepts like continuity,
    differentiability, analyticity and their relations would remain
    at best vaguely perceived. It is only in the study of more general
    functions that one sees the importance of these notions, and
    their different roles, even for the simplest case.
}{%
    \footnotesize
    \textsc{Jon Barwise},
    \textit{Model-theoretic logics}, 1985, \cite{barwise1985background}
}
\endgroup

\vskip -0.5cm

Since their introduction more than 250 years ago \cite{bayes1763, laplace1774memoire}, Bayes' rule and Bayesian inference have been cornerstones of modern probability and statistics. Automatic uncertainty quantification and the flexibility provided by the choice of prior have proved to be attractive properties of the Bayesian perspective, and virtually every field of contemporary statistics and machine learning has deep connections with Bayesian inference \cite{ghavamzadeh2015bayesian, wang2020survey, papamarkou24b}. At the same time, advances in the practical applications of classical Bayesian inference have highlighted its limitations. These include (but are not limited to) the sensitivity to prior and likelihood misspecification, and the difficulty of approximating the Bayesian posterior in large-scale models (see \cite{knoblauch2022optimization} and the references therein). 

Fuelled by these limitations, the Bayesian community has developed novel mechanisms that extend the scope of traditional Bayesian inference and generalize Bayesian posteriors. Those generalizations include tempered or cold posteriors \cite{grunwald2012safe, holmes2017assigning,zhang2024cold,harvey2025learning}, Gibbs posteriors, or more generally, posteriors resulting from minimizing PAC-Bayes bounds \cite{catoni2007pac, bissiri2016general, germain2016pac, syring2017gibbs, casado2024pac}, and posteriors obtained through variational methods \cite{beal2003variational, hernandez2016black, blei2017variational}. Each family of generalized posteriors carries out different modifications of the standard Bayesian posterior, while also addressing different limitations. For example, tempered posteriors include a \textit{temperature} parameter that emphasizes the data (or conversely, the prior) in order to mitigate the effects of model misspecification; Gibbs posteriors substitute the negative log-likelihood with an arbitrary loss function that can provide more flexibility and robustness; and variational inference methods provide tractable approximations of the Bayesian posterior by restricting the approximation to a previously established family of distributions  and optimizing its divergence from the posterior.

These departures from classical Bayesian inference have prompted renewed interest in the asymptotic and finite-sample behaviour of generalized posteriors.
In this paper, we study the generalized posteriors obtained by replacing the Kullback--Leibler (KL) divergence with an $\alpha$-divergence in the variational formulation of Bayesian inference \cite{knoblauch2022optimization}. We consider optimization over unrestricted distributions (absolutely continuous with respect to the prior), thereby isolating the effect of the divergence from that of a restricted variational family. Importantly, under suitable regularity conditions, we provide Bernstein--von Mises-style theorems for these generalized posteriors.

Our results reveal two different regimes. For fixed regularization strength, the $\alpha\in(0,1)$ regime yields Student-$t$ limits and posterior widths around the empirical loss minimizer shrinking faster than $n^{-1/2}$, where $n$ is the sample size; whereas $\alpha>1$ yields compactly supported Pearson type-II limits and widths shrinking more slowly. When the empirical minimizer has the typical $\sqrt{n}$ sampling fluctuations around the population minimizer, this mismatch leads to asymptotic undercoverage and overcoverage, respectively, for the credible regions considered here. The KL case $(\alpha= 1)$ separates these regimes and is the unique member of the $\alpha$-divergence family with the usual posterior width exponent and a Gaussian limit. Thus, by studying a broader class of posteriors we are able to put into perspective the specific role played by KL divergence regularization and Gaussian approximations in classical Bayesian inference, along with their implications for the frequentist calibration of credible regions. 

\subsection{Contributions and related work}

\begin{figure}[t]
    \centering

    \begin{subfigure}[t]{0.49\linewidth}
        \centering
        \includegraphics[width=\linewidth]{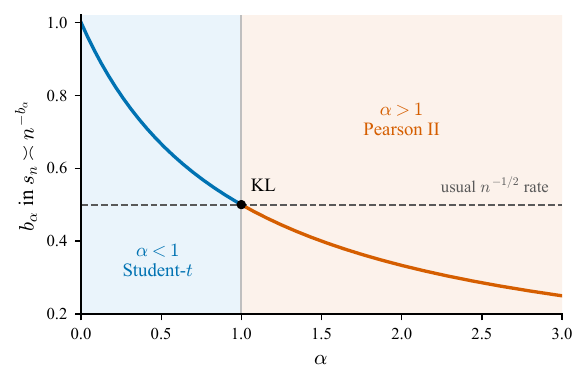}
        \caption{Scale at which each generalized posterior concentrates around the empirical minimizer ($b_\alpha$ in Proposition \ref{prop:alpha_to_one}) as a function of $\alpha$ for $d=1$.}
        \label{fig:first}
    \end{subfigure}
    \hfill
    \begin{subfigure}[t]{0.49\linewidth}
        \centering
        \includegraphics[width=\linewidth]{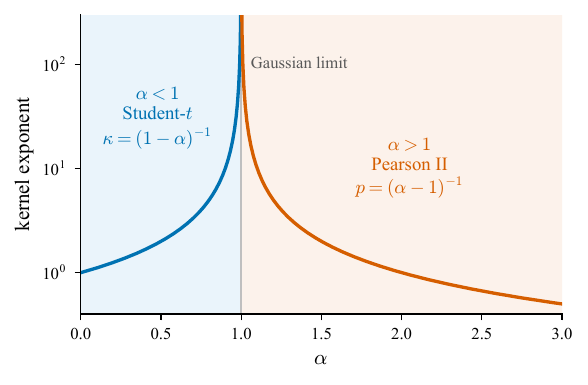}
        \caption{Kernel exponent for the limiting Student-$t$ distribution ($\alpha<1$) and Pearson type-II distribution ($\alpha>1$).}
        \label{fig:second}
    \end{subfigure}

    \caption{For $\alpha$-divergence regularized generalized variational inference, $\alpha=1$ separates two qualitatively different asymptotic regimes, both in terms of local shape and concentration exponent.}
    \label{fig:summary_combined}
\end{figure}

The asymptotic analysis of Bayesian methods can be traced back to Laplace \cite{laplace1774memoire, stigler1986laplace}, but the first \textit{modern} consistency results are those of \cite{doob1949application} and \cite{schwartz1965bayes}.
Beyond consistency, a finer description is provided by the Bernstein--von Mises (BvM) theorem \cite{bernstein1917probability, vonmises1919fundamentalsaetze,lecam1953asymptotic}, which also links Bayesian credible regions with frequentist confidence regions. This theorem states that in regular, correctly specified parametric models, the posterior is asymptotically normal under $\sqrt{n}$ rescaling, with covariance given by the inverse Fisher information \cite[Chapter~10]{van2000asymptotic}. Under misspecification, the posterior remains asymptotically normal at the same rate, but its covariance no longer matches the sampling covariance of its centring \cite{kleijn2012bernstein}, and credible regions lose their frequentist interpretation \cite{muller2013risk}.
Analogous results extend beyond likelihood-based posteriors: see \cite{miller2021asymptotic} and the references therein. 

Concentration properties have been established for Gibbs-type posteriors \cite{martin2022direct,syring2023gibbs} and for
tempered or fractional posteriors \cite{bhattacharya2019fractional,
alquier2020}. Since the covariance of a Gibbs posterior is generally miscalibrated, a related line of work tunes the learning rate to restore frequentist coverage \cite{grunwald2012safe, holmes2017assigning,
syring2019calibrating, lyddon2019general}.
Variational approximations have been analysed along similar lines. For KL-based variational inference, \cite{wang2019frequentist} proved a variational
BvM theorem and \cite{zhang2020convergence} derived contraction rates, while \cite{yang2020alpha} studied variational approximations to fractional
posteriors. The $\alpha$-divergences enter variational
inference primarily as the objective measuring the discrepancy between an approximating family and the posterior \cite{minka2005divergence,
li2016renyi, hernandez2016black}, and the asymptotics of the resulting approximations have been studied in \cite{jaiswal2020asymptotic,
bertholom2024asymptotics}. Our setting is different: here the $\alpha$-divergence acts as the regularizer of a Generalized Variational Inference (GVI) problem in the sense of \cite{knoblauch2022optimization}, where the variational problem is posed over all probability measures on $\Theta\subseteq \mathbb{R}^d$ that are absolutely continuous with respect to the prior, rather than over a parametric family.

Consistency of GVI posteriors was established in \cite{knoblauch2019frequentist}, and finite-sample convergence rates were recently derived in \cite{mildner2025rates} under the assumption of bounded divergences. Closed-form solutions for $f$-divergence regularizers are available through convex duality \cite{alquier2021non, daunas2024equivalence, daunas2026empirical,nguyen2026closed}, but their
asymptotic behaviour has not been characterized fully: BvM-type results are lacking for regularizers beyond the KL divergence. 
We address this gap for the $\alpha$-divergence family, and show that, unlike in the KL case, both the limiting shape and the concentration scale of the generalized posterior change.

More specifically, our main contributions are as follows:

\begin{enumerate}
    \item \textbf{Local limiting distributions and concentration scales.}
    Under suitable local regularity, identifiability, and prior
    assumptions similar to those in \cite{miller2021asymptotic}, we establish convergence in $L^1$ (hence in total variation) of the appropriately centred and rescaled posteriors.
    For $\alpha\in(0,1)$ with $2/(1-\alpha)>d$, the limiting distribution is of Student-$t$ type; for $\alpha>1$, it is a compactly supported Pearson type-II distribution. In both regimes, the posterior width around the empirical loss minimizer is of order

    \begin{equation}
        s_n\asymp n^{-b_\alpha},
        \qquad
        b_\alpha=\frac{1}{2+d(\alpha-1)}.
    \end{equation}
    with an explicit constant that, unlike in the KL case, depends on the prior density at the population minimizer. Thus, changing the divergence affects both the local limiting shape and the concentration exponent.

    \item \textbf{Consequences for uncertainty quantification.}
    When the empirical loss minimizer has nondegenerate
    $\sqrt{n}$ sampling fluctuations, the posterior width and the sampling uncertainty shrink at different rates.
    We show that this mismatch produces frequentist coverage tending to zero for $\alpha\in(0,1)$ and to one for $\alpha>1$, for the classes of credible regions specified in our results, which include the equal-tailed intervals used in practice.

    \item \textbf{Connection with the KL regime.}
    Both families of limiting densities converge in $L^1$ to the Gaussian law as $\alpha\rightarrow1$, while $b_\alpha\rightarrow1/2$.
    This identifies KL as the unique member of the $\alpha$-divergence family with the usual posterior width exponent.
\end{enumerate}

Section \ref{sec:prelim} introduces the generalized variational inference problem and the asymptotic analysis tools used. Section \ref{section:main} presents the main results: Subsections \ref{subsection:<1} and \ref{subsec:alph>1} establish the local limiting distributions, concentration scales, and coverage consequences for $\alpha\in(0,1)$ and $\alpha>1$,
respectively, while Subsection \ref{subsec:KL} describes their connection with the Gaussian limit as $\alpha\rightarrow 1$. Section \ref{section:example} illustrates these results in an intercept-only logistic regression model.
Section \ref{sec:concl} concludes with the limitations of the present analysis and directions for future work. All the proofs are contained in Appendix \ref{Appendix}.

\section{Preliminaries}\label{sec:prelim}

\subsection{Generalized variational inference}

To the best of our knowledge, the most systematic taxonomy of generalized posteriors is provided by \cite{knoblauch2022optimization}. We follow their terminology throughout. Probability measures that are absolutely continuous with respect to Lebesgue measure are identified with their densities whenever convenient.

Let $\Theta\subseteq\mathbb{R}^d$ be a parameter space, let $\Pi$ be a prior distribution on $\Theta$ with density $\pi$, and let $X_1,\ldots,X_n$ be i.i.d.\ observations with distribution $\mathbb{P}$ on a Euclidean space $\mathcal{X}$. Given a loss function
\begin{equation}
    \ell:\Theta\times\mathcal{X}\to\mathbb{R},
\end{equation}
we define the empirical and population losses, respectively, by
\begin{equation}
\ell_n(\theta):=\frac{1}{n}\sum_{i=1}^n\ell(\theta,X_i),\qquad\ell_{\mathbb{P}}(\theta):=\mathbb{E}_{X\sim\mathbb{P}}[\ell(\theta,X)].
\end{equation}
We denote by $\theta_0$ the minimizer of the population loss $\ell_{\mathbb{P}}$, assumed to be unique, and, with a slight abuse of notation, we also write $\mathbb{P}$ for probabilities computed under the sampling distribution of the data.

Following \cite{knoblauch2022optimization}, a generalized posterior is obtained as the solution of a Generalized Variational Inference (GVI) problem of the form
\begin{equation}\label{eq:gvi-general}
Q_n^\star\in\underset{Q\in\mathcal{Q}}{\arg\min}\left\{\int_\Theta \ell_n(\theta)dQ+\frac{\eta D(Q\|\Pi)}{n}\right\},
\end{equation}
where $\mathcal{Q}\subseteq\mathcal{P}(\Theta)$ is a class of probability measures on $\Theta$, $D$ is a divergence on $\mathcal{P}(\Theta)$, and $\eta>0$ controls the relative weight of the divergence term.

\begin{remark}
Observe that the formulation in \eqref{eq:gvi-general} encompasses several standard constructions.

\begin{enumerate}
    \item[(a)] Suppose that $\mathcal{Q}=\mathcal{P}(\Theta)$ and
    $\ell(\theta,x)=-\log p(x\mid\theta)$,
    for some likelihood $p(x\mid\theta)$. Consider also the KL divergence $D=D_{\mathrm{KL}}$, and $\eta=1$. Then \eqref{eq:gvi-general} yields the ordinary Bayesian posterior,
    \[
    \frac{dQ_n^\star}{d\theta}(\theta)
    \propto
    \left\{\prod_{i=1}^n p(X_i\mid\theta)\right\}\pi(\theta).
    \]

    \item[(b)] Under the same choices of loss and divergence, but with
    \[
    \mathcal{Q}\subsetneq\mathcal{P}(\Theta)
    \]
    a restricted, typically parametric, family, \eqref{eq:gvi-general} corresponds to variational inference.

    \item[(c)] If $\mathcal{Q}=\mathcal{P}(\Theta)$, $D=D_{\mathrm{KL}}$, $\eta>0$, and $\ell$ is an arbitrary loss, then the solution is the Gibbs posterior
    \[
    \frac{dQ_n^\star}{d\theta}(\theta)
    \propto
    \exp\left\{-\frac{n}{\eta}\ell_n(\theta)\right\}\pi(\theta).
    \]
\end{enumerate}
\end{remark}

Given this extensive family of novel belief distributions, a natural first question concerns their asymptotic properties.

\subsection{Asymptotic analysis}

Two of the central questions in the asymptotic study of generalized posteriors are consistency and Bernstein--von Mises (BvM) behaviour. Following the terminology of \cite{miller2021asymptotic}, a sequence of generalized posteriors $Q_n^\star$ is consistent at $\theta_0$ if, for every $\varepsilon>0$,
\begin{equation}
Q_n^\star\left(\|\theta-\theta_0\|>\varepsilon\right)
\longrightarrow 0,
\end{equation}
when $n\rightarrow \infty$. Consistency therefore describes concentration of the posterior around the population target, but does not determine either the rate of concentration or the local shape of the posterior.

A BvM theorem provides this finer local description. More precisely, let $\theta_n\to\theta_0$ be a sequence of centring points, which will usually be maximum likelihood estimators or empirical risk minimizers, and let $s_n\to0$ be a concentration scale. If $q_n$ denotes the density of
\begin{equation}
\frac{\theta-\theta_n}{s_n},
\qquad \theta\sim Q_n^\star,
\end{equation}

\noindent a BvM-type result establishes convergence, typically in $L^1$ or total variation,
\begin{equation}
    q_n\longrightarrow q,
\end{equation}

\noindent for some non-degenerate limiting density $q$. Furthermore, observe that consistency itself can also be derived from a BvM theorem. For ordinary Bayesian and Gibbs posteriors under standard regularity conditions, the classical BvM theorem provides $s_n\asymp n^{-1/2}$ and $q$ is Gaussian \cite{van2000asymptotic, miller2021asymptotic}. Our results below show that for $\alpha$-divergence generalized posteriors both features may change: the limiting law need not be Gaussian and the concentration scale need not be $n^{-1/2}$. This will have important consequences for the frequentist calibration of credible regions.

We now make those notions precise. Fix a nominal level $1-\beta\in(0,1)$.
An asymptotically credible sequence of regions is a sequence of data-dependent sets $C_n$ whose posterior
probability satisfies
\begin{equation}
    Q_n^\star(C_n)\longrightarrow 1-\beta
\end{equation}
in probability under the sampling distribution. The frequentist
coverage of $C_n$, however, is the probability that it contains the fixed
population target over repeated realizations of the data: $\mathbb{P}(\theta_0\in C_n)$. More specifically, we say credible regions are asymptotically calibrated at level $1-\beta$
if
\begin{equation}
    \mathbb{P}(\theta_0\in C_n)
    \longrightarrow 1-\beta.
\end{equation}
In this case, they also constitute asymptotically exact frequentist confidence regions.

\section{Main results}\label{section:main}

The variational problem in (\ref{eq:gvi-general}) has been studied for the case of $f$-divergences using techniques from convex duality \cite{alquier2021non,daunas2024equivalence, daunas2026empirical, nguyen2026closed}. More specifically, we will focus on the broad subfamily of $\alpha$-divergences (\ref{eq:alpha-div}), which contains the KL and reverse KL divergences as limiting cases, as well as multiples of the squared Hellinger distance and of the Pearson $\chi^2$ divergence.

In the remaining we set $\Theta=\mathbb R^d$. Consider the unconstrained Generalized Variational Inference (GVI) problem

\begin{equation}\label{eq:GVI_problem}
    Q^\star_n\in \underset{Q\in \mathcal{P}(\mathbb{R}^d)}{\arg\min}\left[ \int_{\mathbb{R}^d} \ell_n(\theta)dQ + \frac{\eta D_\alpha(Q\|\Pi)}{n} \right],
\end{equation}

where $\eta>0$ and for $\alpha\neq 0,1$,

\begin{equation}\label{eq:alpha-div}
D_\alpha(Q\|\Pi)
=\int_{\mathbb{R}^d}f_\alpha\left(\frac{dQ}{d\Pi}\right)d\Pi,\quad\text{with}\quad 
f_\alpha(t)=\frac{t^\alpha-\alpha t+\alpha-1}{\alpha(\alpha-1)}.
\end{equation}

Note that $\alpha\rightarrow 1$ recovers the KL divergence, while $\alpha\rightarrow 0$ converges to the reverse KL divergence. For $\alpha=1/2$, $f_{1/2}(t)=2(\sqrt{t}-1)^2$, so $D_{1/2}(Q\|\Pi)=2\int(\sqrt{dQ}-\sqrt{d\Pi})^2$ is four times the squared Hellinger distance $H^2(Q,\Pi)=\frac{1}{2}\int(\sqrt{dQ}-\sqrt{d\Pi})^2$; for $\alpha=2$, $D_2(Q\|\Pi)=\frac{1}{2}\chi^2(Q\|\Pi)$. 

For simplicity, we restrict (\ref{eq:GVI_problem}) to measures $Q\ll\Pi$, that is, we use the convention $D_\alpha(Q\|\Pi)=+\infty$ when $Q\centernot{\ll}\Pi$. For $\alpha>1$ this agrees with the standard definition of $f$-divergences, because $f_\alpha(t)/t\to\infty$ as $t\to\infty$. For $\alpha\in(0,1)$, however, $f_\alpha(t)/t\to1/(1-\alpha)<\infty$; in fact, $D_\alpha(Q\|\Pi)\leq1/(\alpha(1-\alpha))$ for every $Q\ll\Pi$, and the standard definition would charge the part of $Q$ that is singular with respect to $\Pi$ only the finite amount $Q^\perp(\mathbb{R}^d)/(1-\alpha)$. The restriction is therefore a genuine constraint only when $\alpha\in(0,1)$. 

Under mild assumptions on $\ell$, the solution of (\ref{eq:GVI_problem}) is explicit up to a scalar normalizing constant. Throughout the remainder, we impose assumptions on $\ell_n$ and $\Pi$ analogous to those of Theorem 4 in \cite{miller2021asymptotic}. Note that the assumptions and local-limit results below are formulated for deterministic sequences of loss functions. When the losses depend on random observations, we assume that these conditions hold on a common event of sampling probability one, and in this setting, convergence of the rescaled posterior densities and their posterior probabilities
is understood almost surely. See the discussion in Section 8.8 of \cite{miller2021asymptotic}.

\begin{assumption}[Locally quadratic loss]\label{ass:quadratic}
    There exist $\theta_n\rightarrow\theta_0$, symmetric matrices $H_n\rightarrow H_0$, with $H_0$ positive definite, and functions $R_n$ such that
    \begin{equation}
        \ell_n(\theta)=\ell_n(\theta_n)+\frac{1}{2}(\theta-\theta_n)^\top H_n(\theta-\theta_n)+R_n(\theta-\theta_n),
    \end{equation}
    where for some $\varepsilon_0, c_0>0$, $|R_n(h)|\leq c_0\|h\|^3$, for $\|h\|<\varepsilon_0$.
\end{assumption}

\begin{assumption}[Minimizer identifiability]\label{ass:ident}
    For every $\varepsilon>0$,
    \begin{equation}
        \underset{n\rightarrow \infty}{\lim\inf}\underset{\theta\notin B_\varepsilon(\theta_n)}{\inf}[\Delta_n(\theta)]>0,
    \end{equation}
    where $\Delta_n(\theta):=\ell_n(\theta)-\ell_n(\theta_n).$
\end{assumption}

\begin{assumption}[Prior regularity]\label{ass:prior}
    The prior density $\pi$ is continuous at $\theta_0$ and $\pi(\theta_0)>0$.
\end{assumption}

Note that these assumptions do not involve the $\alpha$-divergence of the GVI problem at all, hence our results will apply to losses and priors such as those discussed in Section 6 of \cite{miller2021asymptotic}. By Lemma \ref{lem:local}, they also imply that, for all sufficiently large $n$, $\theta_n$ is the unique global minimizer of $\ell_n$ and $\Delta_n\geq0$.

\subsection{The $\alpha\in(0,1)$ case}\label{subsection:<1}

We first describe the structure of the minimizer of (\ref{eq:GVI_problem}) for the case where $\alpha\in(0,1)$. For the rest of this section we abbreviate $\kappa:=(1-\alpha)^{-1}$.

\begin{proposition}[Explicit solution for $\alpha\in (0,1)$]\label{prop:exist<1}
Let $\alpha\in(0,1)$ and suppose Assumptions \ref{ass:quadratic}, \ref{ass:ident} and \ref{ass:prior} hold. Assume furthermore that $2\kappa\geq d$.
In that case for all sufficiently large $n$ there exists a unique $\delta_n>0$ such that
\begin{equation}\label{eq:normalizer<1}
\int_{\mathbb{R}^d}\left[\delta_n+\frac{(1-\alpha)n}{\eta}\Delta_n(\theta)\right]^{-\kappa}d\Pi(\theta)=1.
\end{equation}

Furthermore, among $Q\ll\Pi$, the GVI problem in (\ref{eq:GVI_problem}) admits the unique solution
\begin{equation}\label{eq:solution<1}
\frac{dQ_n^\star}{d\Pi}(\theta)=\left[\delta_n+\frac{(1-\alpha)n}{\eta}\Delta_n(\theta)\right]^{-\kappa}.
\end{equation}

\end{proposition}

This explicit form allows us to study its asymptotic behaviour. In particular, we now prove a version of Theorem 4 in \cite{miller2021asymptotic}: a Bernstein--von Mises theorem for $\alpha$-divergence GVI.

\begin{theorem}\label{thm:main_a_leq}
    Suppose Assumptions \ref{ass:quadratic}, \ref{ass:ident} and \ref{ass:prior} hold and $2\kappa>d$. Consider the rescaling factor $s^2_n:=\frac{\eta\delta_n}{n}$. Then the following hold:
    \begin{enumerate}
        \item[(1)] $s_n^{2\kappa-d}(n/\eta)^\kappa\rightarrow C_\alpha:=\pi(\theta_0)\int_{\mathbb{R}^d}\left(1+\frac{1-\alpha}{2}x^\top H_0x\right)^{-\kappa}dx\in(0,\infty)$. In particular, $s_n\asymp n^{-\kappa/(2\kappa-d)}$.
        \item[(2)] If $q_n$ denotes the density of $\frac{\theta-\theta_n}{s_n}$, with $\theta\sim Q^\star_n$, then $q_n\rightarrow q$ in $L^1(\mathbb{R}^d)$, where
        \begin{equation}\label{eq:student-limit}
            q(x)\propto \left(1+\frac{(1-\alpha)}{2}x^\top H_0 x\right)^{-\kappa}.
        \end{equation}
    \end{enumerate}
\end{theorem}

Observe that the limiting distribution in Theorem \ref{thm:main_a_leq}(2) is the multivariate Student-$t$ distribution  $t_{2\kappa-d}\left(0,\frac{2\kappa}{2\kappa-d}H_0^{-1}\right)$, where the second argument denotes the scale matrix, which departs from the classical Bernstein--von Mises behaviour of Bayesian or Gibbs posteriors, where the local shape is Normal.
Our result highlights an important property of $\alpha$-divergence generalized posteriors: unlike standard Bayes or Gibbs posteriors, they show non-Gaussian asymptotics, and the rescaling needed to observe this qualitative behaviour also differs from the standard $n^{-1/2}$. This has important consequences when using these distributions for uncertainty quantification.

Furthermore, part (1) of Theorem \ref{thm:main_a_leq} gives
\[
    s_n\sim C_\alpha^{1/(2\kappa-d)}\left(\frac{\eta}{n}\right)^{\kappa/(2\kappa-d)}.
\]
Hence although $s_n$ is data-dependent, $s_n n^{\kappa/(2\kappa-d)}$ converges to a deterministic constant. Moreover, in contrast with the KL case, where the rescaling $\sqrt{\eta/n}$ does not involve the prior, the asymptotic width of the generalized posterior depends on the prior density at $\theta_0$: a larger $\pi(\theta_0)$ yields a wider posterior. 

Under the standard regularity conditions that guarantee the asymptotic normality of $M$-estimators, $\sqrt{n}(\theta_n-\theta_0)\Rightarrow \mathcal{N}(0,\Sigma) $ \cite{van2000asymptotic}, we observe a scale mismatch: $\theta_n$ fluctuates around $\theta_0$ on the scale $n^{-1/2}$, while Theorem \ref{thm:main_a_leq} asserts that $Q_n^\star$ concentrates around $\theta_n$ at scale $s_n\asymp n^{-\kappa/(2\kappa-d)}$, which is strictly faster than the square root rate when $\alpha\in(0,1)$ and $2\kappa>d.$ This in turn implies that every fixed-level credible region of the form $\theta_n+s_nA$, with $A$ bounded, is eventually too narrow to cover $\theta_0$, making the generalized posterior asymptotically miscalibrated. We formalize the scale mismatch in the following Corollary.

\begin{corollary}[Frequentist coverage]\label{cor:alphaleq1}
Assume the conditions of Theorem~\ref{thm:main_a_leq}, and suppose moreover that
\begin{equation}
    \sqrt{n}(\theta_n-\theta_0)\Rightarrow Z,
\end{equation}
where $\mathbb{P}(Z=0)=0$. Fix the credibility level $1-\beta\in(0,1)$ and let $A_n\subseteq\mathbb{R}^d$ be (possibly data-dependent) measurable sets such that $\int_{A_n} q(x)dx\to1-\beta$ in probability and, for some $M<\infty$, $\mathbb{P}(A_n\subseteq B_M(0))\to1$. Then for $C_n:=\theta_n+s_n A_n$, we have 
\begin{equation}
    Q^\star_n(C_n)\longrightarrow 1-\beta\ \text{in probability},\quad\text{while}\quad\mathbb{P}(\theta_0\in C_n)\longrightarrow 0.
\end{equation}
In particular, this holds for $C_n=\theta_n+s_nA$ with a fixed bounded measurable $A$ such that $\int_A q(x)dx=1-\beta$.

\end{corollary}

\subsection{The $\alpha>1$ case}\label{subsec:alph>1}

We now turn to the case where $\alpha>1$, which exhibits qualitatively different behaviour. 

\begin{proposition}[Explicit solution for $\alpha>1$]\label{prop:exist>1}
    Let $\alpha>1$, put $p=(\alpha-1)^{-1}$, and suppose
    Assumptions \ref{ass:quadratic}, \ref{ass:ident} and
    \ref{ass:prior} hold.
    In that case, for all sufficiently large $n$, there
    exists a unique $\delta_n>0$ such that
    \begin{equation}
        \int_{\mathbb{R}^d}
        \left[
            \delta_n-\frac{(\alpha-1)n}{\eta}
            \Delta_n(\theta)
        \right]^p_+d\Pi(\theta)=1.
    \end{equation}

    Furthermore, the GVI problem in
    (\ref{eq:GVI_problem}) admits the unique solution
    \begin{equation}
        \frac{dQ_n^\star}{d\Pi}(\theta)
        =
        \left[
            \delta_n-\frac{(\alpha-1)n}{\eta}
            \Delta_n(\theta)
        \right]^p_+,
    \end{equation}
    where $[u]_+=\max\{u,0\}$.
\end{proposition}

Note that in this case there is no dimensional restriction to the existence of such a solution, and the restriction $Q\ll\Pi$ is implied by the standard definition of $D_\alpha$. We now provide the analogue to Theorem \ref{thm:main_a_leq} for $\alpha>1$.

\begin{theorem}
\label{thm:bvm-alpha-gt1}
Suppose Assumptions \ref{ass:quadratic}, \ref{ass:ident} and \ref{ass:prior} hold, let $p=(\alpha-1)^{-1}$ and define $s_n^2:=\frac{\eta\delta_n}{n}$.
Then the following hold:
\begin{enumerate}
    \item[(1)] $s_n^{2p+d}(n/\eta)^p\rightarrow C_\alpha^{-1}$, where $C_\alpha:=\pi(\theta_0)\int_{\mathbb{R}^d}\left(1-\frac{1}{2p}x^\top H_0x\right)_+^p dx\in(0,\infty)$. In particular, $s_n\asymp n^{-p/(2p+d)}$.

    \item[(2)] If $q_n$ denotes the density of $\frac{\theta-\theta_n}{s_n},$ with $\theta\sim Q_n^\star$,
    then $q_n\longrightarrow q$ in $L^1(\mathbb R^d)$,
    where
    \begin{equation}
    \label{eq:pearson-limit}
        q(x)
        \propto
        \left(
            1-\frac{1}{2p}x^\top H_0x
        \right)_+^p.
    \end{equation}
\end{enumerate}
\end{theorem}

Observe that the limiting distribution in Theorem
\ref{thm:bvm-alpha-gt1}(2) is a multivariate Pearson
type-II distribution supported on the ellipsoid
$\{x\in\mathbb{R}^d:x^\top H_0x\leq2p\}$.
Its mean is zero and its covariance matrix is
$\frac{2p}{2p+d+2}H_0^{-1}$.
As before, part (1) gives $s_n\sim C_\alpha^{-1/(2p+d)}(\eta/n)^{p/(2p+d)}$; here a larger $\pi(\theta_0)$ yields a narrower posterior, the opposite of the $\alpha\in(0,1)$ case.
Theorem \ref{thm:bvm-alpha-gt1} exhibits the opposite scale
mismatch to the case $\alpha\in(0,1)$. Again, under standard regularity conditions for $M$-estimators,
\begin{equation}
    \sqrt{n}(\theta_n-\theta_0)
    \Rightarrow \mathcal{N}(0,\Sigma).
\end{equation}
Thus, the sampling fluctuations of $\theta_n$ occur at
scale $n^{-1/2}$, whereas Theorem \ref{thm:bvm-alpha-gt1}
shows that $Q_n^\star$ concentrates around $\theta_n$
at scale
\begin{equation}
    s_n\asymp n^{-p/(2p+d)},
    \qquad n^{-1/2}=o(s_n).
\end{equation}
The posterior therefore concentrates around its empirical centre more slowly than that centre concentrates around $\theta_0$. Consequently, fixed-level credible balls centred at $\theta_n$ have asymptotic frequentist coverage one. The following corollary formalizes this scale mismatch for regions containing a neighbourhood of the origin in the rescaled coordinates.

\begin{corollary}[Frequentist coverage]
\label{cor:coverage-alpha-gt1}
Assume the conditions of Theorem~\ref{thm:bvm-alpha-gt1},
and suppose moreover that
\begin{equation}
    \sqrt n(\theta_n-\theta_0)\Rightarrow Z,
\end{equation}
where $Z$ is an $\mathbb{R}^d$-valued random vector.
Fix the credibility level $1-\beta\in(0,1)$ and let $A_n\subseteq\mathbb R^d$ be (possibly data-dependent) measurable sets such that $\int_{A_n} q(x)\,dx\to1-\beta$ in probability and, for some $r>0$, $\mathbb P(\overline B_r(0)\subseteq A_n)\to1$.
Then for $C_n:=\theta_n+s_nA_n$, we have
\begin{equation}
    Q_n^\star(C_n)\longrightarrow1-\beta\ \text{in probability},\quad\text{while}\quad\mathbb P(\theta_0\in C_n)\longrightarrow1.
\end{equation}
In particular, this holds for $C_n=\theta_n+s_nA$ with a fixed measurable $A$ containing a neighbourhood of the origin and such that $\int_A q(x)\,dx=1-\beta$.
\end{corollary}

The two regimes therefore depart from the classical Bernstein--von Mises picture in opposite directions. For $\alpha\in(0,1)$ the generalized posterior concentrates around $\theta_n$ faster than the sampling fluctuations of its centre and consequently undercovers, whereas for $\alpha>1$ it
concentrates more slowly and consequently overcovers. At the boundary
$\alpha=1$, both concentration exponents approach the usual
parametric exponent $1/2$.

\subsection{The phase transition at $\alpha=1$}
\label{subsec:KL}

It is well known that for $\alpha=1$, the GVI posterior is the Gibbs posterior
\begin{equation}
    \frac{dQ_n^\star}{d\Pi}(\theta)
    \propto
    \exp\left\{-\frac{n}{\eta}\ell_n(\theta)\right\},
\end{equation}
whose rescaled density under
$\sqrt{n/\eta}(\theta-\theta_n)$ converges to that of
$\mathcal{N}(0,H_0^{-1})$ under the same local quadratic,
identifiability and prior regularity assumptions; Theorem 4 of \cite{miller2021asymptotic}.
The common scaling $s_n^2=\eta\delta_n/n$ used in the
two non-KL regimes gives the following connection with
this classical BvM behaviour.

\begin{proposition}[Gaussian limit as $\alpha\rightarrow 1$]
\label{prop:alpha_to_one}
    Let $q_\alpha$ denote the limiting densities in
    Theorems \ref{thm:main_a_leq} and
    \ref{thm:bvm-alpha-gt1}, and let $\varphi_{H_0}$ be
    the density of $\mathcal{N}\left(0,H_0^{-1}\right)$.
    Then
    \begin{equation}
        \|q_\alpha-\varphi_{H_0}\|_{L^1(\mathbb{R}^d)}
        \underset{\alpha\rightarrow 1}{\longrightarrow}0.
    \end{equation}
    Moreover, the concentration exponents in both
    theorems can be written as
    \begin{equation}
        s_n\asymp n^{-b_\alpha},
        \qquad
        b_\alpha=\frac{1}{2+d(\alpha-1)},
    \end{equation}
    and $b_\alpha\rightarrow1/2$ as $\alpha\rightarrow1$.
\end{proposition}

It is remarkable that $\alpha=1$ acts as a phase transition between two qualitatively distinct asymptotic regimes: for $\alpha\in(0,1)$ the limiting distribution is a Student-$t$ and the posteriors overconcentrate around $\theta_n$, while for $\alpha>1$ the limiting distribution is a Pearson type-II and the posterior underconcentrates.
Both the limiting shape and the concentration exponent approach their classical Gaussian counterparts when $\alpha\rightarrow 1$. So do the prior-dependent prefactors $C_\alpha^{1/(2\kappa-d)}$ and $C_\alpha^{-1/(2p+d)}$ of Theorems \ref{thm:main_a_leq} and \ref{thm:bvm-alpha-gt1}, which tend to one because $C_\alpha$ converges to a positive constant (see (\ref{eq:kernel-integral})) while the exponents vanish. The KL case is thus the unique member of the $\alpha$-divergence family which exhibits classical BvM behaviour and whose width scale matches the scale of the fluctuations of $\theta_n$ around $\theta_0$. These findings are summarized in Figure \ref{fig:summary_combined}.

\section{Example: intercept-only logistic regression}\label{section:example}

To illustrate our results empirically, we consider a very simple intercept-only logistic regression model, that is, a Bernoulli model with logit link, with independent observations $X_i\sim\mathrm{Bernoulli}(p_0)$, $p_0=0.7$, and parameter $\theta\in\mathbb R$, with success probability $\sigma(\theta)=(1+e^{-\theta})^{-1}$. The empirical loss is the average negative log-likelihood
\begin{equation}
    \ell_n(\theta)=\frac{1}{n}\sum_{i=1}^n\left[\log\left(1+e^\theta\right)-X_i\theta\right]=\log(1+e^\theta)-\widehat p_n\theta,
\end{equation}
where $\widehat p_n=n^{-1}\sum_{i=1}^nX_i$, and the population loss is $\ell_{\mathbb{P}}(\theta)=\log(1+e^\theta)-p_0\theta$. On the event $\{0<\widehat p_n<1\}$, whose probability $1-p_0^n-(1-p_0)^n$ tends to one exponentially fast, the minimizer of $\ell_n$ is $\theta_n=\operatorname{logit}(\widehat{p}_n)$, while the population minimizer is $\theta_0=\operatorname{logit}(p_0)$. Since $\ell_n''(\theta)=\sigma(\theta)(1-\sigma(\theta))$, the corresponding curvatures are $H_n=\widehat p_n(1-\widehat p_n)$ and $H_0=p_0(1-p_0)$.

By the strong law of large numbers, we have $\theta_n\rightarrow\theta_0$ and $H_n\rightarrow H_0$ almost surely, and by the central limit theorem and the delta method,
\begin{equation}
        \sqrt n(\theta_n-\theta_0)
    \Rightarrow\mathcal N\left(0,H_0^{-1}\right).
\end{equation}

The bounded third derivative of the logistic loss ensures the
required local cubic remainder bound in Assumption \ref{ass:quadratic}. Strict convexity and
$\widehat p_n\to p_0$ also imply separation of the empirical loss from its minimum outside any fixed neighbourhood of $\theta_n$.
Together with a continuous, positive prior density, these verify Assumptions \ref{ass:quadratic}, \ref{ass:ident} and \ref{ass:prior} almost surely.

\begin{figure}[t]
    \centering

    \begin{subfigure}[t]{0.48\linewidth}
        \centering
        \includegraphics[width=\linewidth]{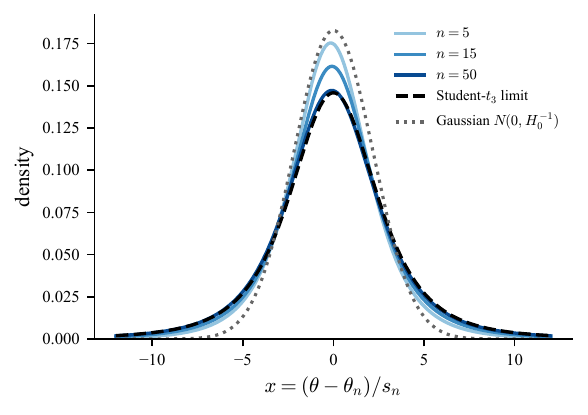}
        \caption{As $n$ increases, the density of $(\theta-\theta_n)/s_n$ approaches the predicted Student-$t_3$ limit.}
        \label{fig:first_alpha_leq}
    \end{subfigure}
    \hfill
    \begin{subfigure}[t]{0.48\linewidth}
        \centering
        \includegraphics[width=\linewidth]{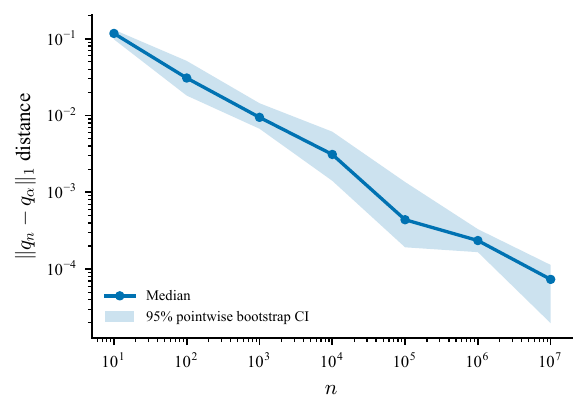}
        \caption{Median $L^1$ distance of $q_n$ to the limiting density across 12 independent sample paths.
        }
        \label{fig:second_alpha_leq}
    \end{subfigure}

    \caption{Convergence of the $s_n$-rescaled $\alpha$-divergence generalized posterior to a Student-$t$ distribution in the logistic regression example of Section \ref{section:example}, with $\alpha=0.5$, $d=1$, and $\eta=1$.}
    \label{fig:combined}
\end{figure}

\begin{figure}[t]
    \centering

    \begin{subfigure}[t]{0.48\linewidth}
        \centering
        \includegraphics[width=\linewidth]{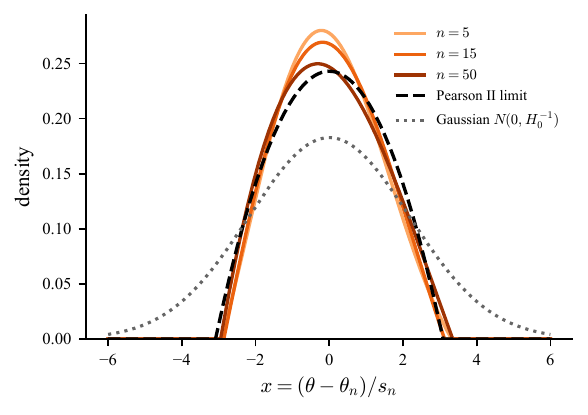}
        \caption{As $n$ increases, the density of $(\theta-\theta_n)/s_n$ approaches the predicted Pearson type-II limit.}
        \label{fig:first_alpha_geq}
    \end{subfigure}
    \hfill
    \begin{subfigure}[t]{0.48\linewidth}
        \centering
        \includegraphics[width=\linewidth]{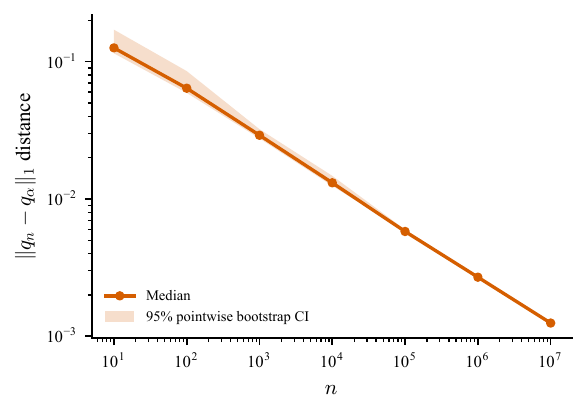}
        \caption{Median $L^1$ distance of $q_n$ to the limiting density across 12 independent sample paths.
        }
        \label{fig:second_alpha_geq}
    \end{subfigure}

    \caption{Convergence of the $s_n$-rescaled $\alpha$-divergence generalized posterior to a Pearson type-II distribution in the logistic regression example of Section \ref{section:example}, with $\alpha=2$, $d=1$, and $\eta=1$.}
    \label{fig:combined_pearson}
\end{figure}

We set $\eta=1$ and use a Student-$t_5$ prior centred at zero with scale parameter $2$.
For the experiments, we consider both $\alpha=0.5$ and $\alpha=2$; the latter corresponds to $D_2=\frac{1}{2}\chi^2$.
For each sample size and divergence, we evaluate the explicit generalized posterior $Q_n^\star$, determining $\delta_n$ numerically.
We then compute the density $q_n$ of $\frac{\theta-\theta_n}{s_n}$, $\theta\sim Q_n^\star$, with $s_n^2=\frac{\eta\delta_n}{n}$.
The limiting laws for the two choices of $\alpha$ are given by
\begin{equation}
    q_{1/2}\ \text{is the density of}\ t_3\left(0,\frac{4}{3H_0}\right)
    \quad\text{and}\quad
    q_2(x)=\frac34\sqrt{\frac{H_0}{2}}
        \left[1-\frac{H_0x^2}{2}\right]_+,
\end{equation}
respectively, where the second argument of the Student
distribution denotes its squared scale (the scale matrix of Theorem \ref{thm:main_a_leq} for $d=1$). The corresponding posterior widths then satisfy $s_n\asymp n^{-2/3}$ and $s_n\asymp n^{-1/3}$.

Figures \ref{fig:first_alpha_leq} and \ref{fig:first_alpha_geq} display the rescaled posterior densities at $n=5,15,50$, using prefixes of one simulated Bernoulli sequence.
We also plot the limiting densities together with $\mathcal N(0,H_0^{-1})$ for comparison. These illustrate how the finite-sample posterior shapes approach their distinct non-Gaussian limits.
To assess convergence across datasets, Figures \ref{fig:second_alpha_leq} and \ref{fig:second_alpha_geq} report median $L^1$ errors along twelve independent sample paths.
Finally, Figure \ref{fig:coverage_combined} examines frequentist coverage of equal-tailed $95\%$ credible intervals computed from the finite-sample posteriors. If $F_{n,k}$ denotes the posterior CDF given $k$ successes, the interval contains $\theta_0$ precisely when
\begin{equation}
    0.025\leq F_{n,k}(\theta_0)\leq0.975.
\end{equation}
We compute coverage exactly, summing over the sampling probabilities of $k\sim\mathrm{Binomial}(n,p_0)$ at different sample sizes. The Gibbs posterior with the same prior and $\eta=1$, that is, the ordinary Bayesian posterior, provides a baseline.

\begin{figure}[t]
    \centering

    \begin{subfigure}[t]{0.48\linewidth}
        \centering
        \includegraphics[width=\linewidth]{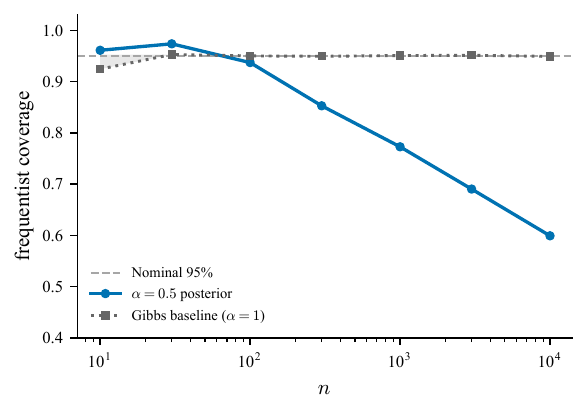}
        \caption{Frequentist coverage of the central $95\%$ $\alpha$-divergence GVI credible interval with $\alpha=0.5$.} 
        \label{fig:coverage_leq}
    \end{subfigure}
    \hfill
    \begin{subfigure}[t]{0.48\linewidth}
        \centering
        \includegraphics[width=\linewidth]{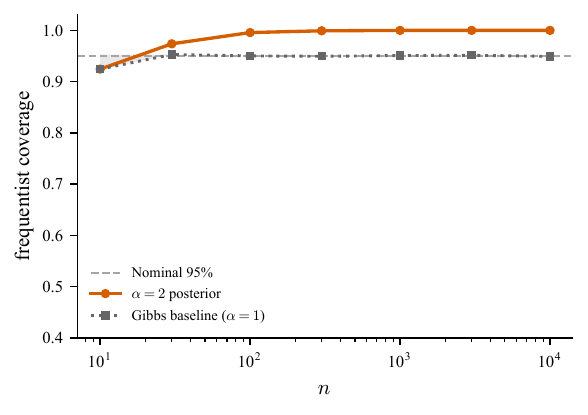}
        \caption{Frequentist coverage of the central $95\%$ $\alpha$-divergence GVI credible interval with $\alpha=2$.}
        \label{fig:coverage_geq}
    \end{subfigure}

    \caption{Frequentist coverage of equal-tailed 95\% posterior intervals for $\alpha=0.5$ (left) and $\alpha=2$ (right), with a Gibbs baseline ($\alpha=1$). The dashed line marks nominal coverage. The curves illustrate the undercoverage and overcoverage predicted by Corollaries \ref{cor:alphaleq1} and \ref{cor:coverage-alpha-gt1} respectively.}
    \label{fig:coverage_combined}
\end{figure}

The coverage curves illustrate the consequences of the scale mismatch established by our corollaries: for $\alpha=0.5$, posterior intervals become too narrow relative to the sampling fluctuations of their centre, whereas for $\alpha=2$ they become too wide. The undercoverage for $\alpha=0.5$ develops slowly, since $\sqrt{n}\,s_n\asymp n^{-1/6}$: the coverage is still about $0.6$ at $n=10^4$, and slightly above nominal for the smallest sample sizes. For $\alpha=2$, $\sqrt{n}\,s_n\asymp n^{1/6}$ and the coverage approaches one quickly. This occurs in a correctly specified model satisfying the regularity assumptions, illustrating that the choice of divergence alone can disrupt the frequentist interpretation of posterior uncertainty.

\section{Conclusions and future work}\label{sec:concl}

We have studied the asymptotic behaviour of unrestricted generalized variational inference posteriors regularized by an $\alpha$-divergence, with fixed regularization strength $\eta>0$. Under local quadraticity, identifiability, and prior regularity assumptions analogous to those used in classical Bernstein--von Mises theory, we show that the divergence determines both the local limiting shape and the concentration scale around the empirical loss minimizer.

For $\alpha\in(0,1)$ satisfying $2/(1-\alpha)>d$, the rescaled posterior converges in $L^1$ to a Student-$t$ distribution, and its width around $\theta_n$ shrinks faster than $n^{-1/2}$. For $\alpha>1$, the limiting distribution is a compactly supported Pearson type-II law, and the posterior width shrinks more slowly. In both cases the asymptotic width depends on the prior density at $\theta_0$, unlike in the KL case. As $\alpha\to 1$, both families of limiting densities converge in $L^1$ to the Gaussian law, while their concentration exponents approach $1/2$. Thus, within the regimes considered here, KL regularization is the unique member of the $\alpha$-divergence family that yields the usual posterior width exponent and Gaussian limit.

These differences have direct consequences for uncertainty quantification. When the empirical minimizer has the usual Gaussian sampling limit under $\sqrt n$ rescaling, its sampling fluctuations and the posterior width occur on different scales for every fixed $\alpha\neq 1$. For the classes of credible regions considered here, this mismatch leads to frequentist coverage tending to zero when $\alpha<1$ and to one when $\alpha>1$, even in regular, correctly specified models. These findings call for caution when using alternative divergences to construct GVI posteriors, since posterior concentration alone does not guarantee calibrated uncertainty quantification.

Several important questions remain open. For $\alpha\in(0,1)$, the regime $2\kappa\leq d$, where $\kappa=(1-\alpha)^{-1}$, behaves differently. A finer description of this setting is left for future work. Another direction is to allow sample-size or data-dependent regularization $\eta=\eta_n$, and to determine whether it can correct the scale mismatch and restore frequentist coverage. Extending the analysis to other $f$-divergences could also clarify which properties of the regularizer determine the limiting shape and concentration scale.
Finally, analysing the variational problem over a restricted family of distributions would help us connect the general theory to applications.

\subsection*{Acknowledgments}

The authors are grateful to Enric Alberola for the insightful and stimulating discussions that enriched this work.

\bibliographystyle{alpha}
\bibliography{refs}

\newpage

\appendix
\section{Proofs for the main results}\label{Appendix}

We first collect the local consequences of the assumptions that are used repeatedly in the following technical lemma.

\begin{lemma}\label{lem:local}
Suppose Assumptions \ref{ass:quadratic} and \ref{ass:ident} hold. There exist $\varepsilon_1\in(0,\varepsilon_0)$, constants $0<c\leq C<\infty$, $\rho>0$ and $N\in\mathbb{N}$ such that, for all $n\geq N$,
\begin{equation}
    c\|h\|^2\leq\Delta_n(\theta_n+h)\leq C\|h\|^2,\quad\text{for }\|h\|\leq\varepsilon_1.
\end{equation}
Furthermore, $ \Delta_n(\theta)\geq\rho\quad\text{for }\theta\notin B_{\varepsilon_1}(\theta_n).
$
In particular, for $n\geq N$, $\theta_n$ is the unique global minimizer of $\ell_n$ and $\Delta_n\geq0$. If, moreover, Assumption \ref{ass:prior} holds, $\varepsilon_1$ and $N$ can be chosen so that, in addition, $c_\pi\leq\pi(\theta)\leq M$ for all $\theta\in B_{\varepsilon_1}(\theta_n)$ and $n\geq N$, where $c_\pi:=\pi(\theta_0)/2$ and $M:=2\pi(\theta_0)$.
\end{lemma}

\begin{proof}
Let $\lambda>0$ and $\Lambda<\infty$ be the smallest and largest eigenvalues of $H_0$. Since $H_n\to H_0$, there is $N_1$ such that $\lambda_{\min}(H_n)\geq\lambda/2$ and $\|H_n\|\leq2\Lambda$ for $n\geq N_1$. Choose $\varepsilon_1\in(0,\varepsilon_0)$ with $c_0\varepsilon_1\leq\lambda/8$. For $\|h\|\leq\varepsilon_1$, Assumption \ref{ass:quadratic} gives $\Delta_n(\theta_n+h)=\frac{1}{2}h^\top H_nh+R_n(h)$ with $|R_n(h)|\leq c_0\|h\|^3\leq\frac{\lambda}{8}\|h\|^2$, so that $\frac{\lambda}{8}\|h\|^2\leq\Delta_n(\theta_n+h)\leq(\Lambda+\frac{\lambda}{8})\|h\|^2$ for $n\geq N_1$. Assumption \ref{ass:ident} with $\varepsilon=\varepsilon_1$ provides $\rho>0$ and $N_2$ such that $\Delta_n\geq\rho$ outside $B_{\varepsilon_1}(\theta_n)$ for $n\geq N_2$; take $N=\max\{N_1,N_2\}$. For the last claim, use continuity of $\pi$ at $\theta_0$ to reduce $\varepsilon_1$ so that $\pi(\theta_0)/2\leq\pi\leq2\pi(\theta_0)$ on $B_{2\varepsilon_1}(\theta_0)$, recompute $\rho$ and $N_2$ for the new $\varepsilon_1$, and increase $N$ so that $\|\theta_n-\theta_0\|<\varepsilon_1$, hence $B_{\varepsilon_1}(\theta_n)\subseteq B_{2\varepsilon_1}(\theta_0)$, for $n\geq N$.
\end{proof}

Throughout the proofs, $n\geq N$ and $\varepsilon_1,c,C,\rho,c_\pi,M$ are as in Lemma \ref{lem:local}.

\subsection*{Proof of Proposition \ref{prop:exist<1}}
\begin{proof}
We first prove the existence of the  normalizer $\delta_n$. Define
\begin{equation}
    F_n(\delta)=\int_{\mathbb{R}^d}\left[\delta+\frac{(1-\alpha)n}{\eta}\Delta_n(\theta)\right]^{-\kappa}d\Pi(\theta),\qquad\delta>0.
\end{equation}

Since $\Delta_n\geq0$, the integrand is bounded by $\delta^{-\kappa}$ and strictly decreasing in $\delta$; hence $F_n$ is finite, continuous and strictly decreasing, with
$F_n(\delta)\to0$ as $\delta\to\infty$. By Lemma \ref{lem:local}, we have that $\Delta_n(\theta_n+h)\leq C\|h\|^2$ and $\pi(\theta_n+h)\geq c_\pi>0$ for $\|h\|\leq\varepsilon_1$, hence by the monotone convergence theorem

\begin{equation}
    F_n(\delta)\rightarrow F_n(0)\geq c_\pi\left[\frac{(1-\alpha)nC}{\eta}\right]^{-\kappa}\int_{\|h\|\leq\varepsilon_1}\|h\|^{-2\kappa}\,dh=\infty\qquad\text{as }\delta\downarrow 0,
\end{equation}
because $\int_{\|h\|\leq\varepsilon_1}\|h\|^{-2\kappa}\,dh$ diverges when $2\kappa\geq d$. Thus, there is a unique $\delta_n>0$ such that
$F_n(\delta_n)=1$.

Now, write $r=dQ/d\Pi$. Up to the constant $n\ell_n(\theta_n)$ and the factor $n$, the minimization objective (\ref{eq:GVI_problem}) is
\begin{equation}
\mathcal J_n(r)=n\int_{\mathbb{R}^d} \Delta_n r\,d\Pi+\eta\int_{\mathbb{R}^d} f_\alpha(r)\,d\Pi.
\end{equation}

Let
\begin{equation}
    r_n(\theta)=\left[\delta_n+\frac{(1-\alpha)n}{\eta}\Delta_n(\theta)\right]^{-\kappa}.
\end{equation}

Since $f_\alpha'(t)=\frac{t^{\alpha-1}-1}{\alpha-1}$,
$r_n$ satisfies the first-order optimality condition
\begin{equation}\label{eq:foc<1}
    n\Delta_n(\theta)+\eta f_\alpha'(r_n(\theta))=\lambda_n:=\frac{\eta(1-\delta_n)}{1-\alpha}.
\end{equation}

Hence, by convexity of $f_\alpha$, for every
probability density $r$ with $\mathcal J_n(r)<\infty$,
\begin{equation}
    \mathcal J_n(r)-\mathcal J_n(r_n)\geq\int_{\mathbb{R}^d}\left[n\Delta_n+\eta f_\alpha'(r_n)\right](r-r_n)\,d\Pi=\lambda_n\int_{\mathbb{R}^d}(r-r_n)\,d\Pi=0.
\end{equation}
Strict convexity of $f_\alpha$ gives uniqueness.
\end{proof}

\subsection*{Proof of Theorem \ref{thm:main_a_leq}}
\begin{proof}
    Recall that $\Delta_n(\theta)=\ell_n(\theta)-\ell_n(\theta_n)$.
    We first show that $\delta_n\rightarrow 0$. Otherwise, along a
    subsequence, $\delta_n\geq\delta_*>0$. Fix
    $\varepsilon\in(0,\varepsilon_1]$. By Assumption \ref{ass:ident},
    there exists $\rho_\varepsilon>0$ such that, for all
    sufficiently large $n$, $\Delta_n(\theta)\geq\rho_\varepsilon$ on $B_\varepsilon(\theta_n)^c$.

    Splitting the normalization integral therefore gives
    \begin{equation}
         1\leq\delta_*^{-\kappa}\Pi(B_\varepsilon(\theta_n))+
        \left[
            \frac{(1-\alpha)n\rho_\varepsilon}{\eta}
        \right]^{-\kappa}.
    \end{equation}

    Since $\theta_n\to\theta_0$, eventually
    $B_\varepsilon(\theta_n)\subseteq B_{2\varepsilon}(\theta_0)$.
    Taking the limit along the subsequence yields
    \begin{equation}
        1\leq\delta_*^{-\kappa}\Pi(B_{2\varepsilon}(\theta_0)).
    \end{equation}

    Letting $\varepsilon\rightarrow 0$ gives a contradiction
    because $\Pi$ has a density. Thus $\delta_n\to 0$,
    and consequently $s_n\to0$.
    
    \vskip 0.5cm

    Now consider the change of variables $\theta=\theta_n+s_nx$. By the definition of $s_n$, we have
    \begin{equation}\label{eq:change_var}
        q_n(x)=s^d_n\delta_n^{-\kappa}\pi(\theta_n+s_nx)\left[1+(1-\alpha)\frac{\Delta_n(\theta_n+s_nx)}{s^2_n}\right]^{-\kappa}.
    \end{equation}
    Assumption \ref{ass:quadratic} then implies that, for $\|s_nx\|<\varepsilon_0$,
    \begin{equation}
        \frac{\Delta_n(\theta_n+s_nx)}{s_n^2}=\frac{1}{2}x^\top H_nx+\frac{R_n(s_nx)}{s^2_n},
    \end{equation}
    and by the cubic bound on $R_n$ we have, for every $x\in\mathbb{R}^d$,
    \begin{equation}
        \frac{|R_n(s_nx)|}{s_n^2}\leq c_0 s_n\|x\|^3\underset{n\rightarrow\infty}{\longrightarrow} 0.
    \end{equation}
    Hence combining the observations above with Assumption \ref{ass:prior} we have, for every fixed $x$,
    \begin{equation}\label{eq:density_limit}
        \pi(\theta_n+s_nx)\left[1+(1-\alpha)\frac{\Delta_n(\theta_n+s_nx)}{s^2_n}\right]^{-\kappa}\underset{n\rightarrow\infty}{\longrightarrow} \pi(\theta_0)\left[1+\frac{(1-\alpha)}{2}x^\top H_0x\right]^{-\kappa}.
    \end{equation}

    We now determine $s^d_n\delta_n^{-\kappa}$. The strategy will be similar to that in Theorem 4 of \cite{miller2021asymptotic}.

    First of all, since $q_n(x)$ is a density, we have $\int q_n(x)dx=1$, hence
    \begin{equation}\label{eq:decomposition}
    \begin{aligned}
        s^{-d}_n\delta^\kappa_n&=\int_{\mathbb{R}^d}\pi(\theta_n+s_nx)\left[1+(1-\alpha)\frac{\Delta_n(\theta_n+s_nx)}{s^2_n}\right]^{-\kappa}dx\\
        &=\int_{\|s_n x\|\leq \varepsilon_1}\pi(\theta_n+s_nx)\left[1+(1-\alpha)\frac{\Delta_n(\theta_n+s_nx)}{s^2_n}\right]^{-\kappa}dx\\ &+ \int_{\|s_n x\|>\varepsilon_1}\pi(\theta_n+s_nx)\left[1+(1-\alpha)\frac{\Delta_n(\theta_n+s_nx)}{s^2_n}\right]^{-\kappa}dx.
    \end{aligned}
    \end{equation}

    For the first term, Lemma \ref{lem:local} gives $\Delta_n(\theta_n+s_nx)\geq cs^2_n\|x\|^2$, therefore 
    \begin{equation}
                \mathbf{1}_{\{\|s_nx\|\leq\varepsilon_1\}}\pi(\theta_n+s_nx)\left[1+(1-\alpha)\frac{\Delta_n(\theta_n+s_nx)}{s_n^2}\right]^{-\kappa}
        \leq M(1+c(1-\alpha)\|x\|^2)^{-\kappa},
    \end{equation}
    where we also used that $\pi$ is bounded by $M$ on $B_{\varepsilon_1}(\theta_n)$. The right-hand side of the inequality is integrable because $2\kappa> d$, so the dominated convergence theorem applies and we have
    \begin{multline}
        \int_{\|s_n x\|\leq \varepsilon_1}\pi(\theta_n+s_nx)\left[1+(1-\alpha)\frac{\Delta_n(\theta_n+s_nx)}{s^2_n}\right]^{-\kappa}dx\\
        \underset{n\rightarrow\infty}{\longrightarrow}\pi(\theta_0)\int_{\mathbb{R}^d}\left[1+\frac{(1-\alpha)}{2}x^\top H_0x\right]^{-\kappa}dx.
    \end{multline}

    We now control the second term in (\ref{eq:decomposition}):

    \begin{align} \int_{\|s_nx\|>\varepsilon_1} \pi(\theta_n+s_nx) \left[ 1+ (1-\alpha)\frac{\Delta_n(\theta_n+s_nx)}{s_n^2} \right]^{-\kappa}dx
    &\leq \left(1+(1-\alpha)\frac{\rho}
    {s_n^2}\right)^{-\kappa}\int_{\mathbb{R}^d}\pi(\theta_n+s_nx)dx\\ &= s_n^{-d}
    \left(1+(1-\alpha)\frac{\rho}{s_n^2}\right)^{-\kappa}\underset{n\rightarrow\infty}{\longrightarrow} 0,
    \end{align}

    where the inequality is an application of Lemma \ref{lem:local}, the equality follows from undoing the change of variables inside the integral, and the limit follows from $2\kappa>d$, since the right-hand side is at most $((1-\alpha)\rho)^{-\kappa}s_n^{2\kappa-d}$.

    In conclusion, combining both limits in (\ref{eq:decomposition}) we obtain
    \begin{equation}\label{eq:constant}
        s^{-d}_n\delta^\kappa_n\underset{n\rightarrow\infty}{\longrightarrow} \pi(\theta_0)\int_{\mathbb{R}^d}\left[1+\frac{(1-\alpha)}{2}x^\top H_0x\right]^{-\kappa}dx=C_\alpha.
    \end{equation}

    Finally, combining (\ref{eq:constant}) and (\ref{eq:density_limit}) with (\ref{eq:change_var}) we obtain $q_n\rightarrow q$ pointwise. Since $q_n$ and $q$ are probability densities, Scheffé's lemma provides the $L^1(\mathbb{R}^d)$ convergence.

    For the first claim, substitute $\delta_n=ns_n^2/\eta$ in (\ref{eq:constant}): since $s_n^{-d}\delta_n^\kappa=s_n^{2\kappa-d}(n/\eta)^\kappa$, we obtain $s_n^{2\kappa-d}(n/\eta)^\kappa\rightarrow C_\alpha$, and hence $\delta_n\asymp n^{-d/(2\kappa-d)}$ and $s_n\asymp n^{-\kappa/(2\kappa-d)}$.
\end{proof}

\subsection*{Proof of Corollary \ref{cor:alphaleq1}}

\begin{proof}

For the first part, note that
\begin{equation}
    |Q^\star_n(C_n)-(1-\beta)|\leq\int_{A_n}|q_n(x)-q(x)|dx+\left|\int_{A_n}q(x)dx-(1-\beta)\right|\longrightarrow 0
\end{equation}
in probability, where the limit follows from Theorem \ref{thm:main_a_leq} and the assumption on $A_n$.
Now, on the event $\{A_n\subseteq B_M(0)\}$, $\theta_0\in C_n$ implies $\|\theta_n-\theta_0\|\leq Ms_n$. Hence
\begin{equation}    
\mathbb{P}(\theta_0\in C_n)\leq\mathbb{P}(A_n\not\subseteq B_M(0))+\mathbb{P}\left(\|\sqrt{n}(\theta_n-\theta_0)\|\leq M\sqrt{n}s_n\right).
\end{equation}

Therefore, for every $\varepsilon>0$,
\begin{equation}\label{eq:inter_corol}
\begin{split}
\mathbb{P}(\theta_0\in C_n)\leq{}&\mathbb{P}(A_n\not\subseteq B_M(0))\\
&+\mathbb{P}\left(\|\sqrt{n}(\theta_n-\theta_0)\|\leq \varepsilon\right) + \mathbb{P}\left(M\sqrt{n}s_n > \varepsilon
\right).
\end{split}
\end{equation}

By Theorem \ref{thm:main_a_leq}, $s_n\asymp n^{-\kappa/(2\kappa-d)},$
and since $2\kappa>d$, we have $\sqrt{n}s_n\asymp n^{-d/(2(2\kappa-d))}\longrightarrow 0$. This in turn makes the last term in \eqref{eq:inter_corol} vanish, the first one vanishes by assumption, and taking $\limsup_{n\to\infty}$, Portmanteau's theorem gives
\begin{equation}
    \limsup_{n}\,\mathbb{P}(\theta_0\in C_n)\leq\mathbb{P}(\|Z\|\leq \varepsilon).
\end{equation}
Finally, letting $\varepsilon\to 0$ and using $\mathbb{P}(Z=0)=0$ proves the second claim.
\end{proof}

\subsection*{Proof of Proposition \ref{prop:exist>1}}
\begin{proof}
    We first prove the existence of the normalizer
    $\delta_n$. For all sufficiently large $n$,
    Lemma \ref{lem:local}
    implies that $\Delta_n\geq0$. Define
    \begin{equation}
        F_n(\delta)
        =
        \int_{\mathbb{R}^d}
        \left[
            \delta-\frac{(\alpha-1)n}{\eta}
            \Delta_n(\theta)
        \right]^p_+d\Pi(\theta),
        \qquad \delta>0.
    \end{equation}

    Since $\Delta_n\geq 0$, the integrand is bounded above by $\delta^p$. Thus $F_n$ is finite and continuous by the dominated convergence theorem, and
    \begin{equation}
        0\leq F_n(\delta)\leq\delta^p
        \underset{\delta\rightarrow0}{\longrightarrow}0.
    \end{equation}

    By Lemma \ref{lem:local}, we have
    $\Delta_n(\theta_n+h)\leq C\|h\|^2$ and
    $\pi(\theta_n+h)\geq c_\pi>0$ whenever
    $\|h\|\leq\varepsilon_1$. Consequently,
    \begin{equation}
        F_n(\delta)\geq\Pi(B_{\varepsilon_1}(\theta_n))
        \left[
            \delta-\frac{(\alpha-1)n}{\eta}
            C\varepsilon_1^2
        \right]^p_+\underset{\delta\rightarrow\infty}{\longrightarrow}\infty.
    \end{equation}

    Moreover, $F_n$ is strictly increasing. Indeed, for
    any $0<\delta_1<\delta_2$, the integrand is
    nondecreasing everywhere and strictly increasing
    on a sufficiently small ball around $\theta_n$
    where
    \begin{equation}
        \frac{(\alpha-1)n}{\eta}\Delta_n(\theta)
        <\delta_1.
    \end{equation}
    This ball has positive prior mass. Thus there is a
    unique $\delta_n>0$ such that $F_n(\delta_n)=1$.

    Now, write $r=dQ/d\Pi$. Up to the constant
    $n\ell_n(\theta_n)$ and the factor $n$, the minimization objective
    (\ref{eq:GVI_problem}) is
    \begin{equation}
        \mathcal J_n(r)
        =
        n\int_{\mathbb{R}^d}\Delta_n r\,d\Pi
        +
        \eta\int_{\mathbb{R}^d}f_\alpha(r)\,d\Pi.
    \end{equation}
    Let
    \begin{equation}
        r_n(\theta)
        =
        \left[
            \delta_n-\frac{(\alpha-1)n}{\eta}
            \Delta_n(\theta)
        \right]^p_+.
    \end{equation}
    This density has finite objective: $r_n$ is bounded,
    $\Delta_n$ is bounded on $\{r_n>0\}$, and $f_\alpha$
    is finite and continuous on $[0,\delta_n^p]$.

    Since
    $f_\alpha'(t)=(t^{\alpha-1}-1)/(\alpha-1)$,
    including the right derivative at $t=0$, we have
    \begin{equation}
        n\Delta_n(\theta)+\eta f_\alpha'(r_n(\theta))
        =
        \frac{\eta(\delta_n-1)}{\alpha-1}
        \qquad\text{on }\{r_n>0\}.
    \end{equation}
    On the other hand, on $\{r_n=0\}$,
    \begin{equation}
        n\Delta_n(\theta)
        \geq\frac{\eta\delta_n}{\alpha-1},
    \end{equation}
    and therefore
    \begin{equation}
        n\Delta_n(\theta)+\eta f_\alpha'(0)
        =
        n\Delta_n(\theta)-\frac{\eta}{\alpha-1}
        \geq\frac{\eta(\delta_n-1)}{\alpha-1}.
    \end{equation}

    Hence, by convexity of $f_\alpha$, for every
    probability density $r$ with finite objective,
    \begin{equation}
        \begin{aligned}
            \mathcal J_n(r)-\mathcal J_n(r_n)
            &\geq
            \int_{\mathbb{R}^d}
            \left[
                n\Delta_n+\eta f_\alpha'(r_n)
            \right](r-r_n)\,d\Pi\\
            &=
            \frac{\eta(\delta_n-1)}{\alpha-1}
            \int_{\mathbb{R}^d}(r-r_n)\,d\Pi
            +
            \int_{\{r_n=0\}}
            \left[
                n\Delta_n-\frac{\eta\delta_n}{\alpha-1}
            \right]r\,d\Pi
            \geq0,
        \end{aligned}
    \end{equation}
    where we used that both densities integrate to one
    and that the last integrand is nonnegative.
    Strict convexity of $f_\alpha$ gives uniqueness.
\end{proof}

\subsection*{Proof of Theorem \ref{thm:bvm-alpha-gt1}}
\begin{proof}
    The structure of the proof is the same as Theorem \ref{thm:main_a_leq}, but we need to take into account the different kernel in this posterior.  We first show that $\delta_n/n\rightarrow 0$. Otherwise,
    along a subsequence, $\delta_n\geq an$ for some $a>0$.
    By Lemma \ref{lem:local},
    \begin{equation}
        \Delta_n(\theta_n+h)\leq C\|h\|^2\qquad\text{for }\|h\|\leq\varepsilon_1.
    \end{equation}
    Fix $\varepsilon\in(0,\varepsilon_1]$ such that
    $\frac{(\alpha-1)C\varepsilon^2}{\eta}\leq\frac{a}{2}$.
    Restricting the normalization integral to
    $B_\varepsilon(\theta_n)$ therefore gives
    \begin{equation}
            1
            =\int_{\mathbb{R}^d}
            \left[
                \delta_n-\frac{(\alpha-1)n}{\eta}
                \Delta_n(\theta)
            \right]^p_+d\Pi(\theta)
            \geq
            \Pi(B_\varepsilon(\theta_n))
            \left(\frac{an}{2}\right)^p.
    \end{equation}
    Since $\theta_n\to\theta_0$, Assumption \ref{ass:prior}
    implies that $\Pi(B_\varepsilon(\theta_n))$ is bounded
    below by a positive constant for all sufficiently large
    $n$. The right-hand side therefore tends to infinity,
    a contradiction. Thus $\delta_n/n\to0$, and consequently
    $s_n\to0$.

    \vskip 0.5cm

    Now consider the change of variables
    $\theta=\theta_n+s_nx$. By the definition of $s_n$,
    we have
    \begin{equation}\label{eq:change_var_gt}
        q_n(x)=s_n^d\delta_n^p\pi(\theta_n+s_nx)
        \left[
            1-(\alpha-1)
            \frac{\Delta_n(\theta_n+s_nx)}{s_n^2}
        \right]^p_+.
    \end{equation}
    Assumption \ref{ass:quadratic} then implies that, for $\|s_nx\|<\varepsilon_0$,
    \begin{equation}
        \frac{\Delta_n(\theta_n+s_nx)}{s_n^2}
        =\frac{1}{2}x^\top H_nx
        +\frac{R_n(s_nx)}{s_n^2},
    \end{equation}
    and by the cubic bound on $R_n$ we have, for every
    $x\in\mathbb{R}^d$,
    \begin{equation}
        \frac{|R_n(s_nx)|}{s_n^2}
        \leq c_0s_n\|x\|^3
        \underset{n\rightarrow\infty}{\longrightarrow}0.
    \end{equation}
    Hence, combining the observations above with
    Assumption \ref{ass:prior} we have, for every fixed $x$,
    \begin{equation}\label{eq:density_limit_gt}
        \pi(\theta_n+s_nx)
        \left[
            1-(\alpha-1)
            \frac{\Delta_n(\theta_n+s_nx)}{s_n^2}
        \right]^p_+
        \underset{n\rightarrow\infty}{\longrightarrow}
        \pi(\theta_0)
        \left[
            1-\frac{\alpha-1}{2}x^\top H_0x
        \right]^p_+.
    \end{equation}

    We now determine $s_n^d\delta_n^p$.
    First of all, since $q_n(x)$ is a density, we have
    $\int q_n(x)dx=1$, hence
    \begin{equation}\label{eq:decomposition_gt}
        \begin{aligned}
            s_n^{-d}\delta_n^{-p}
            &=\int_{\mathbb{R}^d}
            \pi(\theta_n+s_nx)
            \left[
                1-(\alpha-1)
                \frac{\Delta_n(\theta_n+s_nx)}{s_n^2}
            \right]^p_+dx\\
            &=\int_{\|s_nx\|\leq\varepsilon_1}
            \pi(\theta_n+s_nx)
            \left[
                1-(\alpha-1)
                \frac{\Delta_n(\theta_n+s_nx)}{s_n^2}
            \right]^p_+dx\\
            &\quad+\int_{\|s_nx\|>\varepsilon_1}
            \pi(\theta_n+s_nx)
            \left[
                1-(\alpha-1)
                \frac{\Delta_n(\theta_n+s_nx)}{s_n^2}
            \right]^p_+dx.
        \end{aligned}
    \end{equation}

    For the first term, Lemma \ref{lem:local} gives
    $\Delta_n(\theta_n+s_nx)\geq cs_n^2\|x\|^2$
    on $\{\|s_nx\|\leq\varepsilon_1\}$, therefore
    \begin{equation}
        \mathbf{1}_{\{\|s_nx\|\leq\varepsilon_1\}}
        \pi(\theta_n+s_nx)
        \left[
            1-(\alpha-1)
            \frac{\Delta_n(\theta_n+s_nx)}{s_n^2}
        \right]^p_+
        \leq M\left[1-c(\alpha-1)\|x\|^2\right]^p_+,
    \end{equation}
    where we also used that $\pi$ is bounded by $M$
    on $B_{\varepsilon_1}(\theta_n)$. The right-hand side
    is bounded and compactly supported, hence integrable.
    Thus the dominated convergence theorem applies and
    we have
    \begin{equation}
            \int_{\|s_nx\|\leq\varepsilon_1}
            \pi(\theta_n+s_nx)
            \left[
                1-(\alpha-1)
                \frac{\Delta_n(\theta_n+s_nx)}{s_n^2}
            \right]^p_+dx\underset{n\rightarrow\infty}
            {\longrightarrow}
            \pi(\theta_0)\int_{\mathbb{R}^d}
            \left[
                1-\frac{\alpha-1}{2}x^\top H_0x
            \right]^p_+dx.
    \end{equation}

    We now control the second term in
    (\ref{eq:decomposition_gt}).
    By Lemma \ref{lem:local},
    $\Delta_n(\theta)\geq\rho$
    on $B_{\varepsilon_1}(\theta_n)^c$. Since $s_n\to0$,
    \begin{equation}
        1-\frac{(\alpha-1)\rho}{s_n^2}<0
    \end{equation}
    for all sufficiently large $n$. Consequently,
    \begin{equation}
        \int_{\|s_nx\|>\varepsilon_1}
        \pi(\theta_n+s_nx)
        \left[
            1-(\alpha-1)
            \frac{\Delta_n(\theta_n+s_nx)}{s_n^2}
        \right]^p_+dx=0
    \end{equation}
    for all sufficiently large $n$.

    In conclusion, combining both limits in
    (\ref{eq:decomposition_gt}) we obtain
    \begin{equation}\label{eq:constant_gt}
        s_n^{-d}\delta_n^{-p}
        \underset{n\rightarrow\infty}{\longrightarrow}
        \pi(\theta_0)\int_{\mathbb{R}^d}
        \left[
            1-\frac{\alpha-1}{2}x^\top H_0x
        \right]^p_+dx=C_\alpha,
    \end{equation}
    where the limit is finite and strictly positive.

    Finally, combining (\ref{eq:constant_gt}) and
    (\ref{eq:density_limit_gt}) with
    (\ref{eq:change_var_gt}) we obtain $q_n\rightarrow q$
    pointwise. Since $q_n$ and $q$ are probability
    densities, Scheff\'e's lemma provides the
    $L^1(\mathbb{R}^d)$ convergence.

    For the first claim, substituting $\delta_n=ns_n^2/\eta$ gives
    \begin{equation}
        s_n^{-d}\delta_n^{-p}
        =s_n^{-(2p+d)}\left(\frac{n}{\eta}\right)^{-p},
    \end{equation}
    so that $s_n^{2p+d}(n/\eta)^p\rightarrow C_\alpha^{-1}$. Hence
    $\delta_n\asymp n^{d/(2p+d)}$ and $s_n\asymp n^{-p/(2p+d)}$.
\end{proof}

\subsection*{Proof of Corollary \ref{cor:coverage-alpha-gt1}}
\begin{proof}
For the first part, note that
\begin{equation}
    \left|Q_n^\star(C_n)-(1-\beta)\right|
    \leq
    \int_{A_n}|q_n(x)-q(x)|\,dx
    +\left|\int_{A_n}q(x)\,dx-(1-\beta)\right|
    \longrightarrow0\ \text{in probability},
\end{equation}
where the limit follows from
Theorem~\ref{thm:bvm-alpha-gt1} and the assumption on $A_n$.
Now, on the event $\{\overline B_r(0)\subseteq A_n\}$,
$\theta_0\notin C_n$ implies $\|\theta_n-\theta_0\|>rs_n$. Hence
\begin{equation}
    \mathbb P(\theta_0\notin C_n)
    \leq
    \mathbb P(\overline B_r(0)\not\subseteq A_n)+\mathbb P\left(
        \|\sqrt n(\theta_n-\theta_0)\|>r\sqrt n\,s_n
    \right).
\end{equation}

By Theorem~\ref{thm:bvm-alpha-gt1},
$s_n\asymp n^{-p/(2p+d)}$, and since $d\geq1$,
\begin{equation}
    \sqrt n\,s_n
    \asymp n^{d/(2(2p+d))}
    \longrightarrow\infty.
\end{equation}
Therefore, for every $M>0$,
\begin{align}
    \mathbb P(\theta_0\notin C_n)
    &\leq
    \mathbb P(\overline B_r(0)\not\subseteq A_n)+\mathbb P\left(
        \|\sqrt n(\theta_n-\theta_0)\|>M
    \right)
    +\mathbb P\left(r\sqrt n\,s_n\leq M\right).
\end{align}
Taking $\limsup_{n\to\infty}$ and applying the Portmanteau theorem to the closed set $\{\|z\|\geq M\}$ gives
\begin{equation}
    \limsup_{n\to\infty}
    \mathbb P(\theta_0\notin C_n)
    \leq \mathbb P(\|Z\|\geq M).
\end{equation}
Letting $M\to\infty$ proves the claim.
\end{proof}

\subsection*{Proof of Proposition \ref{prop:alpha_to_one}}
\begin{proof}
    Write $z(x)=x^\top H_0x/2$. The unnormalized kernels
    of $q_\alpha$ are
    \begin{equation}
        K_\alpha(x)=
        \begin{cases}
            (1+z(x)/\kappa)^{-\kappa},
            & \alpha<1,\quad \kappa=(1-\alpha)^{-1},\\
            [1-z(x)/p]^p_+,
            & \alpha>1,\quad p=(\alpha-1)^{-1}.
        \end{cases}
    \end{equation}
    Both converge pointwise to $e^{-z(x)}$ as
    $\alpha\rightarrow 1$.

    For $\alpha<1$, fix $\kappa_0>d/2$. Whenever
    $\kappa\geq\kappa_0$,
    \begin{equation}
        K_\alpha(x)
        \leq(1+z(x)/\kappa_0)^{-\kappa_0},
    \end{equation}
    which is integrable. For $\alpha>1$, we have
    $K_\alpha(x)\leq e^{-z(x)}$.
    Dominated convergence therefore gives
    \begin{equation}\label{eq:kernel-integral}
        \int_{\mathbb{R}^d}K_\alpha(x)\,dx
        \underset{\alpha\rightarrow1}{\longrightarrow}
        \int_{\mathbb{R}^d}e^{-z(x)}\,dx
        =\frac{(2\pi)^{d/2}}{\sqrt{\det H_0}}.
    \end{equation}
    Consequently, the normalized densities converge
    pointwise to $\varphi_{H_0}$, and Scheff\'e's lemma
    yields the $L^1(\mathbb{R}^d)$ convergence.
    Finally,
    \begin{equation}
        \frac{\kappa}{2\kappa-d}
        =\frac{1}{2+d(\alpha-1)}
        \quad(\alpha<1),
        \qquad
        \frac{p}{2p+d}
        =\frac{1}{2+d(\alpha-1)}
        \quad(\alpha>1),
    \end{equation}
    proving the assertion about the concentration exponents.
\end{proof}

\end{document}